\documentclass[final]{siamart}
\usepackage{amsfonts}
\newcommand{\TheTitle}{Alpert multiwavelets and Legendre-Angelesco multiple orthogonal polynomials} 
\newcommand{\TheAuthors}{Jeffrey S. Geronimo, Plamen Iliev, Walter Van Assche}

\headers{Multiwavelets \&\ Legendre-Angelesco}{\TheAuthors}

\title{{\TheTitle}\thanks{Version of \today}}

\author{
  Jeffrey S. Geronimo\thanks{School of Mathematics, Georgia Institute of Technology, Atlanta, GA 30332-0160
    (\email{geronimo@math.gatech.edu, iliev@math.gatech.edu}).}
  \and
  Plamen Iliev\footnotemark[2]
  \and
  Walter Van Assche\thanks{Department of Mathematics, KU Leuven, Celestijnenlaan 200B box 2400, BE-3001 Leuven, Belgium
    (\email{walter@wis.kuleuven.be}).}
}

\usepackage{amsopn}
\usepackage{hyphenat}
\ifpdf
\hypersetup{
  pdftitle={\TheTitle},
  pdfauthor={\TheAuthors}
}
\fi

\begin{document}

\maketitle

% REQUIRED
\begin{abstract}
We show that the multiwavelets, introduced by Alpert in 1993, are related to type I Legendre-Angelesco multiple orthogonal polynomials.
We give explicit formulas for these Legendre-Angelesco polynomials and for the Alpert multiwavelets. The multiresolution analysis can
be done entirely using Legendre polynomials, and we give an algorithm, using Cholesky factorization, to compute the
multiwavelets and a method, using the Jacobi matrix for Legendre polynomials, to compute the matrices in the scaling relation for any size of the multiplicity of the multiwavelets.   
\end{abstract}

% REQUIRED

% REQUIRED

\section{Introduction}
Wavelets are a powerful way for approximating functions and have been very successful in harmonic analysis, Fourier analysis,
signal and image processing and approximation theory \cite{Daub}, \cite{Keinert}. Multiresolution analysis, as introduced by Mallat and Meyer
\cite{Mallat,Meyer}, is a powerful way to describe functions in $L^2(\mathbb{R})$ in a nested sequence of subspaces, and wavelets are very well
suited to move from one resolution to higher resolutions. Multiwavelets are an extension of wavelets where instead of using a single
scaling function one uses a vector scaling function, which allows to describe the spaces in the multiresolution analysis
in terms of translates of linear combinations of the functions in the scaling vector. Multiwavelets came up in a natural way
in fractal interpolation and iterated functions theory \cite{DonGerHar, DonGerHarMas, GerHarMas, HarKesMas} and were also put forward by Goodman 
\cite{Goodman94, Goodman93} and Herv\'e \cite{Herve}. A good description of multiwavelets can be found in \cite[part II]{Keinert} and \cite{Keinert0}. 
In this paper we will investigate the multiwavelets described by Alpert \cite{Alpert}. Previously Geronimo, Marcell\'an and Iliev
investigated a modification of Alpert's multiwavelets in \cite{GerMar} and \cite{GerIliev}, but they used an orthogonal basis for the wavelets 
with fewer zero moments than proposed by Alpert in \cite{Alpert}. In this paper we will investigate the wavelets with the maximal number of zero moments, as originally proposed by Alpert. 
 
% The outline is not required, but we show an example here.
The paper is organized as follows. We give some information on Alpert's multiwavelets in
\cref{sec:Alpert}, we make the connection with type I multiple orthogonal polynomials in
\cref{sec:MultOP} and give explicit formulas for the Legendre-Angelesco multiple orthogonal polynomials
in \cref{sec:Leg-Ang}. A formula for Alpert's multiwavelets in terms of the Legendre-Angelesco polynomials
is given in \cref{sec:AlpertLA} including explicit expressions for their Fourier transforms.
An algorithm to expand the multiwavelets in terms of Legendre polynomials is given in \cref{sec:AlpertLeg}.

\section{Alpert multiwavelets}
\label{sec:Alpert}

Let $n \in \mathbb{N} = \{1,2,3,\ldots\}$ be fixed and consider the functions
\[   \phi_j(x) = \frac{1}{\sqrt{2j-1}} P_{j-1}(2x-1) \chi_{[0,1)}(x), \qquad 1 \leq j \leq n, \]
where $P_j$ $(0 \leq j \leq n-1)$ are the Legendre polynomials on $[-1,1]$ and $\chi_{[0,1)}$ is the characteristic function
of the interval $[0,1)$. If we set
\[    \Phi_n(x) = \begin{pmatrix} \phi_1(x) \\ \vdots \\ \phi_n(x)  \end{pmatrix} , \]
then $\Phi_n$ is a vector of compactly supported $L^2$-functions with
\[    \int_0^1 \Phi_n(x)\Phi_n^T(x)\, dx = \mathbb{I}_n, \]
where $\mathbb{I}_n$ is the identity matrix of order $n$. Introduce the linear space
\[    V_0 = \textrm{cl}_{L^2(\mathbb{R})} \textrm{span} \{ \phi(\cdot - j), 1 \leq i \leq n,\ j\in \mathbb{Z} \}, \]
then $V_0$ is a space of piecewise polynomials of degree $\leq n-1$ and it is a finitely generated shift invariant space,
for which $\{ \phi_i(x-j), 1 \leq i \leq n, j \in \mathbb{Z} \}$ is an orthonormal basis. For every $k \in \mathbb{Z}$ we define
$V_k = \{ \phi(2^k \cdot), \phi \in V_0\}$, so that $V_k$ contains functions in $L^2(\mathbb{R})$ at different resolutions.
Clearly the spaces $(V_k)_{k \in \mathbb{Z}}$ are nested
\[   \cdots \subset V_{-2} \subset V_{-1} \subset V_0 \subset V_1 \subset V_2 \subset \cdots \subset L^2(\mathbb{R}), \]
and they form a multiresolution analysis (MRA) of $L^2(\mathbb{R})$. We can write
\[     V_{k+1} = V_k \oplus W_k, \qquad k \in \mathbb{Z}, \]
and the spaces $(W_k)_{k \in \mathbb{Z}}$ are called the wavelet spaces. If $\psi_1,\ldots,\psi_n$ generate a shift-invariant
basis for $W_0$, then these functions are known as multiwavelets when $n > 1$. For $n=1$ one retrieves the MRA associated
with the Haar wavelet. Multiwavelets for $n>1$ were investigated by Alpert \cite{Alpert}, Geronimo and Marcell\'an \cite{GerMar},
Geronimo and Iliev \cite{GerIliev}. 
  
In \cite{Alpert} Alpert constructs functions $f_1,\ldots,f_n$ supported on $[-1,1]$ with the following properties:
\begin{enumerate}
 \item The restriction of $f_i$ to $(0,1)$ is a polynomial of degree $n-1$;  \label{Alpert1}
 \item $f_k(-t) = (-1)^{k+n-1} f_k(t)$ for $t \in (0,1)$;  \label{Alpert2}
 \item $f_1,\ldots,f_n$ have the orthogonality property   \label{Alpert3}
\begin{equation}  \label{fortho}
    \int_{-1}^1 f_i(t)f_j(t)\, dt = \delta_{i,j}, \qquad 1 \leq i,j \leq n;
\end{equation} 
 \item The function $f_k$ has vanishing moments   \label{Alpert4}
\begin{equation} \label{fmoments} 
   \int_{-1}^1 f_k(t)t^i \, dt = 0, \qquad i=0,1,\ldots,k+n-2. 
\end{equation}
\end{enumerate}

Our observation is the following
\begin{proposition}  \label{prop:1}
The last function $f_n$ is, up to a normalization factor, the type I Legendre-Angelesco multiple orthogonal polynomial
\[   f_n(x) = A_{n,n}(x)\chi_{[-1,0)} + B_{n,n}(x) \chi_{[0,1)}.  \]
\end{proposition}

Of course we need to explain what the type I Legendre-Angelesco polynomial is. In the next section we explain what multiple orthogonal polynomials
are and in particular we define the type I and type II Legendre-Angelesco polynomials. \Cref{prop:1} then follows immediately from
the definition of type I Legendre Angelesco polynomials. We will then also explain how the other multiwavelets $f_1,\ldots,f_{n-1}$
can be expressed in terms of the Legendre-Angelesco polynomials.

\section{Multiple orthogonal polynomials}
\label{sec:MultOP}

Let $r \in \mathbb{N}$ and $\vec{n}=(n_1,\ldots,n_r) \in \mathbb{N}^r$ with size $|\vec{n}|=n_1+n_2+\ldots+n_r$. 
Multiple orthogonal polynomials are polynomials in one variable
that satisfy orthogonality conditions with respect to $r$ measures $\mu_1,\ldots,\mu_r$ on the real line. There are two types: 
the type I multiple orthogonal polynomials for the multi-index $\vec{n}$ are $(A_{\vec{n},1},\ldots,A_{\vec{n},r})$, where the degree
of $A_{\vec{n},j}$ is $\leq n_j-1$ and
\[    \sum_{j=1}^r \int x^k A_{\vec{n},j}(x)\, d\mu_j(x) = 0, \qquad 0 \leq k \leq |\vec{n}|-2, \]
with the normalization
\[    \sum_{j=1}^r \int x^{|\vec{n}|-1} A_{\vec{n},j}(x)\, d\mu_j(x) = 1,  \]
and the type II multiple orthogonal $P_{\vec{n}}$ is the monic polynomial of degree $|\vec{n}|$ for which
\[   \int x^k P_{\vec{n}}(x)\, d\mu_j(x) = 0, \qquad 0 \leq k \leq n_j-1, \]
for $1 \leq j \leq r$ (see, e.g., \cite[Chapter 26]{Ismail}). If the measures are supported on disjoint intervals, then the system
$\mu_1,\ldots,\mu_r$ is called an Angelesco system. In this paper we will deal with two measures $(r=2)$ with $\mu_1$ the uniform
measure on $[-1,0]$ and $\mu_2$ the uniform measure on $[0,1]$, and the corresponding multiple orthogonal polynomials are called
Legendre-Angelesco polynomials. These polynomials were introduced by Angelesco in 1918--1919 and investigated in detail by
Kalyagin and Ronveaux \cite{Kalyagin,KalyaRon}.    

The type I Legendre-Angelesco polynomials $(A_{n,m}, B_{n,m})$ are thus defined by $\deg A_{n,m} = n-1$, $\deg B_{n,m} = m-1$ and the
orthogonality conditions
\begin{equation}   \label{Angelortho}
  \int_{-1}^1 \Bigl( A_{n,m}(x) \chi_{[-1,0]} + B_{n,m}(x) \chi_{[0,1]} \Bigr) x^k \, dx = 0, \qquad 0 \leq k \leq n+m-2, 
\end{equation}
and the normalization
\begin{equation}   \label{Angelnorm}
   \int_{-1}^1 \Bigl( A_{n,m}(x) \chi_{[-1,0]} + B_{n,m}(x) \chi_{[0,1]} \Bigr) x^{n+m-1} \, dx = 1.  
\end{equation}
These polynomials are uniquely determined by these conditions.
Note that the orthogonality \cref{Angelortho} corresponds to the condition for the vanishing moments (condition \ref{Alpert4}) for $k=n$,
hence proving \cref{prop:1}. The normalization, however, is different
since one uses condition \ref{Alpert3} instead of \cref{Angelnorm}. Note that the type I Legendre-Angelesco polynomials are piecewise polynomials 
on $[-1,1]$ but their restrictions to $[-1,0]$ and $[0,1]$ are respectively the polynomials $A_{n,m}$ (of degree $n-1$) and $B_{n,m}$ (of degree $m-1$).

The remaining multiwavelets $f_1,\ldots,f_{n-1}$ can also be expressed in terms of Legendre\hyp Angelesco orthogonal polynomials.
Denote the type I multiple orthogonal polynomials by
\[    Q_{n,m}(x) = A_{n,m}(x) \chi_{[-1,0]} + B_{n,m}(x) \chi_{[0,1]}. \]
As mentioned before, these are piecewise polynomials on $[-1,1]$ as are the multiwavelets.

\begin{proposition}  \label{prop:2}
Let $(f_1,\ldots,f_n)$ be the Alpert multiwavelets of multiplicity $n$. Then
\begin{equation*}  \label{fkQ}
   f_k(x) = \begin{cases}  \sum_{j=\frac{k+n}2}^n c_j Q_{j,j}(x), & \textrm{if $k+n$ is even}, \\
                           \frac12 \sum_{j=\frac{k+n+1}2}^n a_j d_{j-1} Q_{j,j}(x) + \sum_{j=\frac{k+n-1}2}^{n-1} d_j Q_{j+1,j}(x), & \textrm{if $k+n$ is odd},
               \end{cases}
\end{equation*}
where the coefficients $(a_j), (c_j), (d_j)$ are real numbers.
\end{proposition}

\begin{proof}
The type I multiple orthogonal polynomials $\{Q_{j,j}, 1 \leq j \leq n\}$ on the diagonal, together with $\{Q_{j+1,j}, 0 \leq j \leq n-1\}$ 
are a basis for the piecewise
polynomials $p_{n-1}\chi_{[-1,0]} + q_{n-1} \chi_{[0,1]}$, where $p_{n-1},q_{n-1} \in \mathbb{P}_{n-1}$ are polynomials of degree at most $n-1$.
There is a biorthogonality relation for the multiple orthogonal polynomials \cite[\S 23.1.3]{Ismail}
\[    \int_{-1}^1 P_{n,m}(x) Q_{k,\ell}(x)\, dx = \begin{cases} 0 & \textrm{if $k \leq n$ and $\ell \leq m$}, \\
                                                                0 & \textrm{if $n+m \leq k+\ell-2$}, \\
                                                                1 & \textrm{if $n+m=k+\ell-1$},  \end{cases}  \]
where $P_{n,m}$ are the type II Legendre-Angelesco multiple orthogonal polynomials, for which
\begin{eqnarray*}
   \int_{-1}^0 P_{n,m}(x) x^k\, dx &=& 0, \qquad 0 \leq k \leq n-1, \\
   \int_0^1 P_{n,m}(x) x^k \, dx  &=& 0, \qquad 0 \leq k \leq m-1. 
\end{eqnarray*}
The biorthogonality gives for the multi-indices near the diagonal
\[    \int_{-1}^1 P_{j,j}(x) Q_{k,k}(x)\, dx = 0 = \int_{-1}^1 P_{j+1,j}(x)Q_{k+1,k}(x)\, dx , \qquad j,k \in \mathbb{N}, \]
\[    \int_{-1}^1 P_{j,j}(x)Q_{k+1,k}(x)\, dx = \delta_{k,j} = \int_{-1}^1 P_{j,j-1}(x)Q_{k,k}(x)\, dx.  \]
So if we expand $f_k$ in the basis of type I Legendre-Angelesco polynomials, then
\begin{equation}  \label{fkQ1}
   f_k(x) = \sum_{j=1}^n c_j Q_{j,j}(x) + \sum_{j=0}^{n-1} d_j Q_{j+1,j}(x), 
\end{equation}
and the biorthogonality gives
\begin{equation}  \label{cdj}
     c_j = \int_{-1}^1 f_k(x) P_{j,j-1}(x)\, dx ,   \quad
     d_j = \int_{-1}^1 f_k(x) P_{j,j}(x)\, dx.   
\end{equation}
The conditions imposed by Alpert give some properties for these coefficients. The symmetry property \ref{Alpert2} gives (change variables $x \to -x$)
\[    d_j = \int_{-1}^1 f_k(-x) P_{j,j}(-x)\, dx = (-1)^{k+n+1} \int_{-1}^1 f_k(x) P_{j,j}(x)\, dx = (-1)^{k+n+1} d_j, \]
so that $d_j = 0$ whenever $k+n$ is even. We have used that the type II multiple orthogonal polynomial $P_{j,j}$ is an even function.
For $c_j$ we have (with the same change of variables)
\[   c_j = \int_{-1}^1 f_k(-x) P_{j,j-1}(-x)\, dx = (-1)^{k+n} \int_{-1}^1 f_k(x) P_{j-1,j}(x)\, dx, \]
where we used that $P_{j,j-1}(-x)=-P_{j-1,j}(x)$, which can easily be seen from the orthogonality conditions.
The type II multiple orthogonal polynomials satisfy the nearest neighbor recurrence relations
\begin{eqnarray*}
   xP_{n,m}(x) = P_{n+1,m}(x) + c_{n,m} P_{n,m}(x) + a_{n,m}P_{n-1,m}(x) + b_{n,m} P_{n,m-1}(x), \\
   xP_{n,m}(x) = P_{n,m+1}(x) + d_{n,m} P_{n,m}(x) + a_{n,m}P_{n-1,m}(x) + b_{n,m} P_{n,m-1}(x). 
\end{eqnarray*}
Subtracting both relations gives
\[   P_{n+1,m}(x)-P_{n,m+1}(x) = (d_{n,m}-c_{n,m}) P_{n,m}(x), \]
so that $P_{j,j-1}(x)-P_{j-1,j}(x) = a_j P_{j-1,j-1}(x)$, where $a_j = d_{j-1,j-1}-c_{j-1,j-1}$. Using this we find
\[    c_j = (-1)^{k+n} \int_{-1}^1 f_k(x) \Bigl( P_{j,j-1}(x) - a_j P_{j-1,j-1}(x) \Bigr)\, dx = (-1)^{k+n} (c_j - a_j d_{j-1}), \]
and hence we have that $2c_j = a_j d_{j-1}$ whenever $k+n$ is odd. So the expansion \cref{fkQ1} reduces to
\begin{equation}  \label{fkQ2}
   f_k(x) = \begin{cases}  \sum_{j=1}^n c_j Q_{j,j}(x), & \textrm{if $k+n$ is even}, \\
                           \frac12 \sum_{j=1}^n a_j d_{j-1} Q_{j,j}(x) + \sum_{j=0}^{n-1} d_j Q_{j+1,j}(x), & \textrm{if $k+n$ is odd}.
               \end{cases}
\end{equation} 
Alpert's moment condition \ref{Alpert4} also shows that 
\[     \begin{cases}
          c_j = 0, & \textrm{if $2j-1 \leq k+n-2$}, \\
          d_j = 0, & \textrm{if $2j \leq k+n-2$}, \end{cases}  \]
so that the expansion \cref{fkQ2} reduces to the one given in the proposition.
\end{proof}

\section{Legendre-Angelesco polynomials}
\label{sec:Leg-Ang}

It is now important to find the Legendre\hyp Angelesco polynomials. We introduce two families of polynomials
\begin{align}   
    p_n(x) &= \sum_{k=0}^n \binom{n}{k} \binom{n+\frac{k}2}{n} (-1)^{n-k} x^k, \label{pn} \\
    q_n(x) &= \sum_{k=0}^n \binom{n}{k} \binom{n+\frac{k-1}2}{n} (-1)^{n-k} x^k. \label{qn}
\end{align}
Observe that $p_n$ and $q_n$ are polynomials of degree $n$ with positive leading coefficient.
These two families are immediately related to the last two multiwavelets, and this holds for every multiplicity $n$.
We will need the Mellin transform of these polynomials to make the connection.

\begin{proposition}  \label{prop:Mellin}
The Mellin transforms of the polynomials $p_n$ and $q_n$ on $[0,1]$ are given by
\begin{equation}   \label{Mellinp}
   \int_0^1 p_n(x) x^s \, dx = (-1)^n \frac{(\frac12 - \frac{s}2)_n}{(s+1)_{n+1}}, 
\end{equation}
and
\begin{equation}   \label{Mellinq}
    \int_0^1 q_n(x) x^s\, dx = (-1)^n \frac{(-\frac{s}{2})_n}{(s+1)_{n+1}}.
\end{equation}
\end{proposition}

\begin{proof}
The proof is along the lines of \cite[Lemma 4]{GerIliev} and \cite[Thm.~2.2]{SmetWVA}.
We will compute the Mellin transform of $p_n$, which by integrating \cref{pn} is
\[   \int_0^1 p_n(x) x^s\, dx = \sum_{k=0}^n \binom{n}{k} \binom{n+\frac{k}2}{n} \frac{(-1)^{n-k}}{k+s+1}.  \]
This is the partial fraction decomposition of a rational function $T(s)/R(s)$ with poles at the integers
$\{-k-1, 0 \leq k \leq n\}$, so $R(s) = (s+1)_{n+1}$ and $T$ is a polynomial of degree $\leq n$ with residue
$(-1)^{n-k} \binom{n}{k} \binom{n+\frac{k}2}{n}$ for the pole at $-k-1$, hence 
\[     \frac{T(-k-1)}{R'(-k-1)} = \binom{n}{k} \binom{n+\frac{k}2}{n} (-1)^{n-k} .  \]
Observe that $R'(-k-1) = (-1)^k k! (n-k)!$ so that
\[    T(-k-1) = (-1)^n \binom{n+\frac{k}2}{n} n! = (-1)^n (\frac{k}{2}+1)_n , \qquad 0 \leq k \leq n.   \]
The polynomial $T(s) = (-1)^n (\frac12 - \frac{s}2)_n$ of degree $n$ has precisely these values at $-k-1$, and hence \cref{Mellinp} follows.
Observe that the Mellin transform of $p_n$ is a rational function of $s$ and that the numerator
has zeros when $s=2k-1$ $(1 \leq k \leq n)$, so that all the odd moments $\leq 2n-1$ of $p_n$ on $[0,1]$ vanish.
The Mellin transform of $q_n$ can be obtained easily in a similar way. From \cref{Mellinq} one sees that $q_n$ has vanishing 
even moments $\leq 2n-2$.
\end{proof}

The last two multiwavelets $f_n$ and $f_{n-1}$ are now easily given in terms of these polynomials.

\begin{proposition}  \label{prop:fpq}
Let $f_1,\ldots,f_n$ be Alpert's multiwavelets of multiplicity $n$. Then for $x \in [0,1]$
\[    f_n(x) = c_{n,0} p_{n-1}(x), \quad   f_{n-1}(x) = d_{n,0} q_{n-1}(x), \]
where $c_{n,0}$ and $d_{n,0}$ are normalizing constants given by
\[ \frac{1}{c_{n,0}^2} = 2 \int_0^1 p_{n-1}^2(x)\, dx, \quad \frac{1}{d_{n,0}^2} = 2 \int_0^1  q_{n-1}^2(x)\, dx.  \]
\end{proposition}

\begin{proof}
We need to verify the moment condition \cref{fmoments} for $f_n$. Observe that
\begin{align*}
   \int_{-1}^1 f_n(x) x^k \, dx &= \int_0^1  f_n(x) x^k\, dx + \int_{-1}^0 f_n(x) x^k\, dx \\
                              &= \bigl(1 + (-1)^{k+1}\bigr) \int_0^1 f_n(x) x^k \, dx,
\end{align*}
so that all the even moments are already zero. So we only need to concentrate on the odd moments of $p_{n-1}$.
The Mellin transform \cref{Mellinp} for $p_{n-1}$ shows that the odd moments $\leq 2n-3$ vanish, and since all the even moments 
for $f_n$ are already zero, it must follow from \cref{fmoments} that $f_n$ is equal to $p_{n-1}$ up to a normalizing factor.

The proof for $f_{n-1}$ is similar with a few changes. The symmetry $f_{n-1}(-x) = f_{n-1}(x)$  now gives
\[  \int_{-1}^1  f_{n-1}(x) x^k \, dx = \bigl(1 + (-1)^{k}\bigr) \int_0^1 f_{n-1}(x) x^k \, dx, \]
so that all the odd moments vanish. We only need to check that the even moments vanish. The orthogonality of $f_n$ and $f_{n-1}$
will automatically be true because of the symmetry:
\[   \int_{-1}^1 f_n(x)f_{n-1}(x)\, dx = \int_0^1 f_n(x)f_{n-1}(x)\, dx - \int_{0}^1 f_n(x)f_{n-1}(x)\, dx = 0. \]
For $q_{n-1}$ we have from \cref{Mellinq} that all the even moments $\leq 2n-4$
vanish, so that together with the vanishing odd moments, the moment condition \cref{fmoments} holds and $f_{n-1}$ must be proportional to $q_{n-1}$.
\end{proof}

An explicit formula for the normalizing constants $c_{n,0}$ and $d_{n,0}$ will be given in \cref{prop:ppqq}.

\subsection{Type I Legendre-Angelesco polynomials}
\label{sec:Leg-AngI}

The type I Legendre\hyp Angelesco polynomials near the diagonal can be expressed explicitly in terms of the functions 
$p_n$ and $q_n$ in \cref{pn}--\cref{qn}:

\begin{proposition} \label{prop:ABpq}
Let $(A_{n,m},B_{n,m})$ be the type I Legendre-Angelesco polynomial for the multi-index $(n,m)$. Then on the diagonal
\begin{equation}   \label{Bp}
   B_{n+1,n+1}(x) =  \frac{1}{2} \frac{(3n+2)!}{n!(2n+1)!} p_n(x), \quad  A_{n+1,n+1}(x) = -B_{n+1,n+1}(-x). 
\end{equation}
and for $|n-m|=1$
\begin{align}
    b_n B_{n+1,n}(x) &= \binom{n+\frac{n}2}{n} q_n(x) - \binom{n+\frac{n-1}2}{n} p_n(x), \label{Bn+1n} \\
    b_n B_{n,n+1}(x) &= \binom{n+\frac{n}2}{n} q_n(x) + \binom{n+\frac{n-1}2}{n} p_n(x), \label{Bnn+1}
\end{align}
and
\begin{equation}  \label{AB}
   A_{n+1,n}(x) = B_{n,n+1}(-x), \quad A_{n,n+1}(x) =  B_{n+1,n}(-x),  
\end{equation}
where the normalizing constant is given by $b_n = 2(\frac{n}2 +1)_n \frac{(2n)!}{(3n+1)!}$.
\end{proposition}

\begin{proof}
The result in \cref{Bp} follows from \cref{prop:1} and \cref{prop:fpq}.
The constant can be found by putting $s=2n+1$ in the Mellin transform, since that moment has to be 1. 
For \cref{Bn+1n}--\cref{AB} one can check that the vanishing moment conditions \cref{Angelortho}
hold and that the polynomials are of the correct degree. The normalizing constant $b_n$ can be obtained
by computing the $2n$th moment of $q_n$, i.e., putting $s=2n$ in \cref{Mellinq}, and the using the normalization
in \cref{Angelnorm}.
\end{proof}

A consequence of the Mellin transform \cref{Mellinp} is the Rodrigues type formula
\[    p_n(x) = \frac{(-1)^n}{2^n n!} (D^*)^n x^{2n}(1-x)^n, \qquad  D^* = \frac{1}{x} \frac{d}{dx}, \]
which follows by using
\[    (s-2k-1) \hat{f}(s) = - \widehat{D_k f}(s), \qquad   D_k f(x) = x^{-2k-1} \frac{d}{dx} x^{2k+2} f(x) , \]
where $\hat{f}$ is the Mellin transform
of $f$
\[     \hat{f}(s) = \int_0^\infty f(x) x^s\, dx,  \]
(see, e.g.,  \cite[\S 2.2]{SmetWVA}). In a similar way the Mellin transform \cref{Mellinq} implies the Rodrigues type formula
\[     q_n(x) = \frac{(-1)^n}{2^n n!} x (D^*)^n x^{2n-1}(1-x)^n. \]

The $p_n$ and $q_n$ can be written as linear combinations of two hypergeometric functions
\begin{multline*}
   p_n(x) = (-1)^n {}_3F_2\left( \begin{array}{c} n+1, - \frac{n}2, -\frac{n}2+\frac12 \\ \frac12, 1 \end{array} ; x^2 \right)  \\
    - (-1)^n \frac{(\frac32)_n}{(n-1)!} x\ {}_3F_2\left( \begin{array}{c} n+\frac32, - \frac{n}2+\frac12, -\frac{n}2+1 \\ \frac32, \frac32 \end{array} ; 
   x^2\right), 
\end{multline*}
and
\begin{multline*}
   q_n(x) = (-1)^n \frac{(\frac12)_n}{n!} {}_3F_2\left( \begin{array}{c} n+\frac12, - \frac{n}2, -\frac{n}2+\frac12 \\ \frac12, \frac12 \end{array} ; 
   x^2\right) \\
    - (-1)^n n x\  {}_3F_2\left( \begin{array}{c} n+1, - \frac{n}2+\frac12, -\frac{n}2+1 \\ 1 , \frac32 \end{array} ; x^2 \right). 
\end{multline*}
There is also a system of recurrence relations. 

\begin{proposition}   \label{prop:rec}
One has
\begin{align}
    (3n-1)x p_{n-1}(x) &= 2n q_n(x) + (2n-1) q_{n-1}(x), \label{xpn} \\
     (3n-2) x q_{n-1}(x) &= \frac23 \Bigl( n p_n(x) + (2n-1) p_{n-1}(x) + (n-1) p_{n-2}(x) \Bigr).  \label{xqn}
\end{align}
\end{proposition}

\begin{proof}
We will prove this by comparing coefficients. For \cref{xpn} the coefficient of $(-x)^k$ on the right hand side is
\[      (-1)^n \left( 2n \binom{n}{k} \binom{n+\frac{k-1}2}{n} - (2n-1) \binom{n-1}{k} \binom{n-1+\frac{k-1}{2}}{n-1} \right).  \]
This can easily be seen to be equal to
\[      (-1)^n \frac{(1+\frac{k-1}{2})_{n-1}}{k!(n-k)!}  \left( 2n(n+\frac{k-1}2) - (2n-1)(n-k) \right).  \]
The factor between brackets is $k(3n-1)$, hence together this gives
\[   (-1)^{n}(3n-1)  \frac{(1+\frac{k-1}{2})_{n-1}}{(k-1)!(n-k)!} = (-1)^{n}(3n-1) \binom{n-1}{k-1} \binom{n-1+\frac{k-1}2}{n-1}, \]
which is the coefficient of $(-x)^k$ of the left hand side of \cref{xpn}.
The proof is similar for \cref{xqn}, except that one now has to combine three terms on the right hand side.
\end{proof} 

These recurrence relations can be simplified in the following way. Introduce one more sequence of polynomials
\begin{equation}  \label{rn}
   r_n(x) = \sum_{k=0}^n \binom{n}{k} \binom{n+\frac{k+1}2}{n} (-1)^{n-k} x^k  , 
\end{equation}
then
\begin{align*}
     xr_{n-1}(x) &= \frac{2}{3} \Bigl( p_{n-1}(x) + p_{n}(x) \Bigr),  \\
     (3n-1)x p_{n-1}(x) &= 2n q_n(x) + (2n-1) q_{n-1}(x), \\
     (3n+1)q_{n}(x) &= (n+1)r_{n}(x) + n r_{n-1}(x),
\end{align*}
or in matrix form
\[  \begin{pmatrix} p_n(x) \\ q_n(x) \\ r_n(x) \end{pmatrix}
     = \begin{pmatrix}  -1 & 0 & \frac32 x \\[4pt]
            \frac{3n-1}{2n} x & - \frac{2n-1}{2n} & 0 \\[4pt]
            \frac{(3n+1)(3n-1)}{2n(n+1)} x & - \frac{(3n+1)(2n-1)}{2n(n+1)} & - \frac{n}{n+1} \end{pmatrix} 
       \begin{pmatrix} p_{n-1}(x) \\ q_{n-1}(x) \\ r_{n-1}(x) \end{pmatrix} .  \]

\subsection{Type II Legendre-Angelesco polynomials}
\label{sec:Leg-AngII}

The type II Legendre\hyp Angelesco polynomials $P_{n,n}(x)$ and $P_{n+1,n}(x)$ also have nice Mellin transforms.

\begin{proposition}  \label{typeIImellin}
One has
\begin{equation}   \label{PnnMellin}
    \int_0^1 P_{n,n}(x) x^s\, dx = C_n \frac{(-s)_n}{(\frac12+\frac{s}2)_{n+1}} 
\end{equation}
and
\begin{equation}   \label{Pn1nMellin}
    \int_0^1 [P_{n+1,n}(x)+P_{n,n+1}(x)] x^s\, dx = D_n \frac{(-s)_n}{(1+\frac{s}2)_{n+1}} ,  
\end{equation}
where $C_n$ and $D_n$ are constants.
\end{proposition}

\begin{proof}
The orthogonality conditions for $P_{n,n}$ are
\[   \int_{-1}^0 P_{n,n}(x) x^k \, dx = 0 = \int_0^1 P_{n,n}(x) x^k\, dx , \qquad 0 \leq k \leq n-1.  \]
The symmetry $P_{n,n}(-x)=P_{n,n}(x)$ implies that $P_{n,n}(x)$ is an even polynomial and we only need to consider
the orthogonality conditions on $[0,1]$. Let
\[   P_{n,n}(x) = \sum_{j=0}^n a_j x^{2j}, \]
then the Mellin transform is
\[    \int_0^1 P_{n,n}(x) x^s\, dx = \sum_{j=0}^n \frac{a_j}{2j+s+1}.  \]
This is a rational function of $s$ with poles at the odd negative integers $-1,-3,-5,\ldots,$ $-2n-1$, 
hence the denominator is $(\frac{1}{2}+\frac{s}2)_{n+1}$.
The orthogonality conditions on $[0,1]$ imply that the numerators has zeros at $0,1,\ldots,n-1$, hence the numerator is proportional to
$(-s)_n$. This gives \cref{PnnMellin}, where the constant $C_n$ has to be determined so that $P_{n,n}$ is a monic polynomial.

The orthogonality conditions for $P_{n+1,n}$ and $P_{n,n+1}$ are
\[    \int_{-1}^0 P_{n+1,n}(x) x^k \, dx = 0 = \int_{0}^1 P_{n,n+1}(x) x^k\, dx , \qquad 0 \leq k \leq n, \]
and
\[    \int_0^1 P_{n+1,n}(x) x^k\, dx = 0 = \int_{-1}^0 P_{n,n+1}(x) x^k\, dx, \qquad 0 \leq k \leq n-1.  \]
The difference $P_{n+1,n}(x)-P_{n,n+1}(x)$ is a polynomial of degree $2n$ which is proportional to $P_{n,n}$, and $P_{n+1,n}(-x) = - P_{n,n+1}(x)$.
The sum $P_{n+1,n}+P_{n,n+1}$ is therefore an odd polynomial and we can write
\[   P_{n+1,n}(x)+P_{n,n+1}(x) = x \sum_{j=0}^n b_j x^{2j}.  \]
The Mellin transform of this sum is of the form
\[   \int_0^1  [P_{n+1,n}(x)+P_{n,n+1}(x)] x^s\, dx = \sum_{j=0}^n \frac{b_j}{2j+s+2}, \]
and this is a rational function with poles at the even negative integers $-2,-4,-6,\ldots,$ $-2n-2$. Hence the denominator is
$(1+\frac{s}2)_{n+1}$. The orthogonality conditions on $[0,1]$ imply that the numerator vanishes at $0,1,\ldots,n-1$, hence it is proportional to
$(-s)_n$. This gives \cref{Pn1nMellin}, where the constant $D_n$ has to be determined so that the leading coefficient is 2.  
\end{proof}

The Mellin transform \cref{PnnMellin} implies the Rodrigues formula
\[    P_{n,n}(x) =  (-1)^n\frac{(2n)!}{(3n)!} \frac{d^n}{dx^n} x^n(1-x^2)^n,  \]
if we take into account that
\[    \int_0^1 (1-x^2)^n x^s\, dx = \frac{1}{2} \frac{n!}{(\frac12+\frac{s}2)_{n+1}}.   \]
This Rodrigues formula was already known, e.g., \cite[\S 23.3]{Ismail}, and it gives the explicit expression
\[    P_{n,n}(x) = \frac{n! (2n)!}{(3n)!} \sum_{k=0}^n \binom{n}{k} \binom{n+2k}{2k} (-1)^{n-k} x^{2k}.  \]

\section{Alpert multiwavelets and Legendre-Angelesco polynomials}
\label{sec:AlpertLA}

\Cref{prop:fpq} already gives the last two multiwavelets $f_n$ and $f_{n-1}$ in terms of the polynomials $p_{n-1}$ and $q_{n-1}$.
In order to be able to work with multiwavelets of various multiplicity, we will change the notation somewhat. If $n \in \mathbb{N}$
we will denote the multiwavelets of multiplicity $n$ by $(f_1^n,f_2^n,\ldots,f_n^n)$, and these correspond to what we so far denoted as
$(f_1,f_2,\ldots,f_n)$. \Cref{prop:fpq} then says that for $x \in [0,1]$
\[      f_{n+1}^{n+1}(x) = c_{n+1,0} p_n(x), \quad f_n^{n+1}(x) = d_{n+1,0} q_n(x).  \]
To find the other multiwavelets, we proceed as follows. Every $f^{n+1}_{n+1-2k}$ is a linear combination of
$p_n, p_{n-1}, \ldots, p_{n-k}$ in such a way that the orthogonality \cref{fortho} holds. This means that we start from the sequence
$(p_n,p_{n-1},\ldots,p_{n-k})$ and we use the Gram-Schmidt process to obtain $(f^{n+1}_{n+1}, f^{n+1}_{n-1}, \ldots, f^{n+1}_{n+1-2k})$. 
In a similar way every $f^{n+1}_{n-2k}$ is a linear combination of $q_n,q_{n-1},q_{n-k}$ in such a way that the orthogonality \cref{fortho} holds.
Hence we use Gram-Schmidt on $(q_n,q_{n-1},\ldots,q_{n-k})$ to obtain $(f^{n+1}_n,f^{n+1}_{n-2},\ldots,f^{n+1}_{n-2k})$.
Observe that the Gram-Schmidt process is applied to the polynomials $(p_n)_n$ and $(q_n)_n$ starting from degree $n$ and going down in the degree,
which is different from the usual procedure where one starts from the lowest degree, going up to the highest degree.
If we denote by $\langle f,g \rangle$ the integral of $fg$ over the interval $[0,1]$, then an explicit formula for $x \in (0,1]$ is
\begin{equation} \label{fn+12k}
   f^{n+1}_{n+1-2k}(x) = c_{n+1,k}  \begin{vmatrix}
           \langle p_n,p_n \rangle & \langle p_n,p_{n-1} \rangle & \cdots & \langle p_n, p_{n-k} \rangle \\
           \langle p_{n-1},p_n \rangle & \langle p_{n-1},p_{n-1} \rangle & \cdots & \langle p_{n-1}, p_{n-k} \rangle \\
             \vdots & \vdots & \cdots & \vdots \\
           \langle p_{n-k+1},p_n \rangle & \langle p_{n-k+1},p_{n-1} \rangle & \cdots & \langle p_{n-k+1},p_{n-k} \rangle \\ 
            p_n(x) & p_{n-1}(x) & \cdots & p_{n-k}(x)   
              \end{vmatrix}, 
\end{equation}
where $c_{n+1,k}$ is such that the norm of $f^{n+1}_{n+1-2k}$ is one. Similarly
\begin{equation}  \label{fn2k}
   f^{n+1}_{n-2k}(x) = d_{n+1,k}  \begin{vmatrix}
           \langle q_n,q_n \rangle & \langle q_n,q_{n-1} \rangle & \cdots & \langle q_n, q_{n-k} \rangle \\
           \langle q_{n-1},q_n \rangle & \langle q_{n-1},q_{n-1} \rangle & \cdots & \langle q_{n-1}, q_{n-k} \rangle \\
             \vdots & \vdots & \cdots & \vdots \\
           \langle q_{n-k+1},q_n \rangle & \langle q_{n-k+1},q_{n-1} \rangle & \cdots & \langle q_{n-k+1},q_{n-k} \rangle \\ 
            q_n(x) & q_{n-1}(x) & \cdots & q_{n-k}(x)   
              \end{vmatrix}, 
\end{equation}
where the constant $d_{n+1,k}$ normalizes the function to have norm one.
For the normalizing constant $c_{n+1,k}$ we define
\begin{equation}   \label{Deltap}
   \Delta^p_{n,k}  =  \begin{vmatrix}
         \langle p_n,p_n \rangle & \langle p_n,p_{n-1} \rangle & \cdots & \langle p_n, p_{n-k} \rangle \\
           \langle p_{n-1},p_n \rangle & \langle p_{n-1},p_{n-1} \rangle & \cdots & \langle p_{n-1}, p_{n-k} \rangle \\
             \vdots & \vdots & \cdots & \vdots \\
           \langle p_{n-k},p_n \rangle & \langle p_{n-k},p_{n-1} \rangle & \cdots & \langle p_{n-k}, p_{n-k} \rangle
          \end{vmatrix}, 
\end{equation}
then it follows from \cref{fn+12k} that $\langle f^{n+1}_{n+1-2k},p_{n-j} \rangle = 0$ for $0 \leq j < k$, since this is equal to a determinant with two
equal rows, and $\langle f^{n+1}_{n+1-2k},p_{n-k} \rangle = c_{n+1,k} \Delta^p_{n,k}$. Now on $[0,1]$ we have $f^{n+1}_{n+1-2k}(x) = c_{n+1,k} \Delta^p_{n,k-1} p_{n-k}(x) + \cdots$, 
and by the symmetry properties of the functions $f_k^{n+1}$
\[   \int_{-1}^1 [f_{n+1-2k}^{n+1}(x)]^2\, dx = 2 \int_0^1 [f_{n+1-2k}^{n+1}(x)]^2 \, dx , \]
hence
\[    \frac12 = \langle f^{n+1}_{n+1-2k},f^{n+1}_{n+1-2k} \rangle = c_{n+1,k}^2 \Delta^p_{n,k}\Delta^p_{n,k-1}, \]
so that 
\begin{equation}   \label{normc}
    c_{n+1,k} = \frac{1}{\sqrt{2\Delta^p_{n,k}\Delta^p_{n,k-1}}}.  
\end{equation}
In a similar way we find
\begin{equation}    \label{normd}
    d_{n+1,k} = \frac{1}{\sqrt{2\Delta^q_{n,k}\Delta^q_{n,k-1}}}, 
\end{equation}
where
\begin{equation}   \label{Deltaq}
   \Delta^q_{n,k}  = \det \begin{pmatrix}
         \langle q_n,q_n \rangle & \langle q_n,q_{n-1} \rangle & \cdots & \langle q_n, q_{n-k} \rangle \\
           \langle q_{n-1},q_n \rangle & \langle q_{n-1},q_{n-1} \rangle & \cdots & \langle q_{n-1}, q_{n-k} \rangle \\
             \vdots & \vdots & \cdots & \vdots \\
           \langle q_{n-k},q_n \rangle & \langle q_{n-k},q_{n-1} \rangle & \cdots & \langle q_{n-k}, q_{n-k} \rangle
          \end{pmatrix}. 
\end{equation}

\subsection{Hypergeometric functions}
We will express the inner products $\langle p_n,p_k \rangle$ in terms of hypergeometric functions.

\begin{proposition}  \label{prop:ppqq}
The entries in the Gram-matrix for $k \leq n$ are given by
\begin{equation}  \label{pnkF}
  \int_0^1 p_n(x)p_k(x)\, dx = (-1)^{n+k} \frac{\Gamma(n+\frac12)}{(n+1)! \sqrt{\pi}}\ 
   {}_4F_3\left( \begin{array}{c} k+1, \frac12 , -\frac{k-1}{2}, -\frac{k}{2} \\ -n+\frac12 , \frac{n+2}2 , \frac{n+3}{2} \end{array} ; 1\right) ,
\end{equation}
and
\begin{equation}   \label{qnkF}
  \int_0^1 q_n(x)q_k(x)\, dx = (-1)^{n+k} \frac{\Gamma(n-\frac12)k}{2(n+2)!\sqrt{\pi}} \ 
    {}_4F_3\left( \begin{array}{c}  k+1, \frac32, -\frac{k-2}{2} , -\frac{k-1}{2} \\ -n + \frac32 , \frac{n+3}2 , \frac{n+4}2 \end{array};1 \right). 
\end{equation}
\end{proposition}

\begin{proof}
For \cref{pnkF} we expand $p_k$ using \cref{pn} to find
\[   \langle p_n,p_k \rangle = (-1)^k \sum_{j=0}^k \binom{k}{j} \binom{k+\frac{j}2}{k} (-1)^j \int_0^1 x^j p_n(x)\, dx.  \]
The integral can be evaluated by using the Mellin transform \cref{Mellinp}, which gives
\[    \int_0^1 x^j p_n(x)\, dx = (-1)^n \frac{(\frac12 - \frac{j}2)_n}{(1+j)_{n+1}}, \]
and this is zero whenever $j$ is odd and $1 \leq j \leq 2n-1$. So we only need to consider even terms in the sum. This leaves
\begin{eqnarray*}
   \langle p_n,p_k \rangle &=& (-1)^{n+k} \sum_{j=0}^{\lfloor \frac{k}{2} \rfloor} \binom{k}{2j} \binom{k+j}{k} 
           \frac{(\frac12-\frac{j}2)_n}{(1+2j)_{n+1}} \\
     &= & (-1)^{n+k} \sum_{j=0}^{\lfloor \frac{k}{2} \rfloor} \frac{(k+j)! (\frac12 - j)_n}{(k-2j)! j! (2j+n+1)!} . 
\end{eqnarray*}
Now we can write
\[   (\frac12 -j)_n = (\frac12)_j \frac{(\frac12)_n}{(-n+\frac12)_j}   \]
and $(k+j)! = (k+1)_j k!$, so that
\[  \langle p_n,p_k \rangle = (-1)^{n+k} (\frac12)_n k!  \sum_{j=0}^{\lfloor \frac{k}{2} \rfloor} 
   \frac{(k+1)_j (\frac12)_j}{(k-2j)!(2j+n+1)! (-n+\frac12)_j j!}.     \]
Next, we have
\[   (k-2j)! =\frac{k!}{(-k)_{2j}} = \frac{k!}{2^{2j} (-\frac{k}{2})_j (-\frac{k-1}2)_j}, \]
and $(2j+n+1)! = 2^{2j} (\frac{n+2}2)_j (\frac{n+3}2)_j (n+1)!$,
so that
\[  \langle p_n,p_k \rangle = (-1)^{n+k} \frac{(\frac12)_n}{(n+1)!}  \sum_{j=0}^{\lfloor \frac{k}{2} \rfloor} 
    \frac{(k+1)_j (-\frac{k}2)_j (-\frac{k-1}2)_j (\frac12)_j}{(-n+\frac12)_j(\frac{n+2}2)_j (\frac{n+3}2)_j j!} , \]
and this coincides with the hypergeometric expression in \cref{pnkF}.
The computations for $\langle q_n,q_k \rangle$ are similar, starting from the expansion \cref{qn} and using the Mellin transform \cref{Mellinq}.
\end{proof}
   
Observe that both \cref{pnkF} and \cref{qnkF} are terminating hypergeometric series.

\subsection{Fourier transform}
\label{sec:Fourier}

We will now compute the Fourier transform of the multiwavelets $f_k^{n+1}$ for $k=n+1$ and $k=n$.
\begin{proposition}  \label{prop:Four}
One has
\[  \hat{f}_{n+1}^{n+1}(t) =  \int_{-1}^1 f^{n+1}_{n+1}(x) e^{ixt}\, dx = 2ic_{n+1,0}t^{2n+1} \frac{(-1)^n n!}{(3n+2)!} \
            {}_1F_2\left( \begin{array}{c}   n+1 \\ \frac{3n+3}2 , \frac{3n+4}2 \end{array}; -\frac{t^2}{4} \right), \]
and
\[   \hat{f}_n^{n+1}(t) =  \int_{-1}^1 f^{n+1}_{n}(x) e^{ixt}\, dx = 2d_{n+1,0}t^{2n} \frac{(-1)^n n!}{(3n+1)!} \
            {}_1F_2\left( \begin{array}{c}   n+1 \\ \frac{3n+2}2 , \frac{3n+3}2 \end{array}; -\frac{t^2}{4} \right). \]
\end{proposition}

\begin{proof}
We have
\begin{eqnarray*}   
  \int_{-1}^1 f_{n+1}^{n+1}(x) e^{ixt}\, dx &=& \int_0^1 f_{n+1}^{n+1}(x) e^{ixt}\, dx - \int_0^1 f_{n+1}^{n+1}(x) e^{-ixt}\, dx  \\
    &=& 2i \int_0^1 f_{n+1}^{n+1}(x) \sin(xt)\, dx . 
\end{eqnarray*}
If we use the Taylor series expansion of the $\sin$ function, then the Fourier transform is
\[ \hat{f}_{n+1}^{n+1}(t) = 2i \sum_{k=0}^{\infty} (-1)^k \frac{t^{2k+1}}{(2k+1)!} \int_0^1 x^{2k+1} f_{n+1}^{n+1}(x)\, dx.  \]
On $(0,1]$ we have $f_{n+1}^{n+1}(x) = c_{n+1,0}p_n(x)$, so we can use the Mellin transform in \cref{Mellinp} to evaluate the integral:
\[    \hat{f}_{n+1}^{n+1}(t) = (-1)^n 2i c_{n+1,0} \sum_{k=0}^{\infty} (-1)^k \frac{t^{2k+1}}{(2k+1)!} \frac{(-k)_n}{(2k+2)_{n+1}} . \]
The first $n$ terms in the series are zero, so if we set $k=j+n$, then
\[  \hat{f}_{n+1}^{n+1}(t) = 2i c_{n+1,0} t^{2n+1} \sum_{j=0}^{\infty} (-1)^j \frac{t^{2j}}{(2j+2n+1)!} \frac{(-j-n)_n}{(2j+2n+2)_{n+1}} . \]
Now use 
\[ (2j+2n+1)! (2j+2n+2)_{n+1} = (2j+3n+2)! = (3n+2)!\ 2^{2j} \bigl( \frac{3n+3}2 \bigr)_j \bigl( \frac{3n+4}2 \bigr)_j, \]
and $(-j-n)_n = (-1)^n n! (n+1)_j/j!$ to find
\[  \hat{f}_{n+1}^{n+1}(t) = 2i c_{n+1,0} t^{2n+1} \frac{(-1)^n n!}{(3n+2)!} 
      \sum_{j=0}^{\infty} (-1)^j \frac{t^{2j}}{2^{2j}} \frac{(n+1)_j}{j! (\frac{3n+3}2)_j (\frac{3n+4}2)_j}, \]
which can be identified in terms of the hypergeometric series.

The proof for the second Fourier transform is similar. We now have
\[   \hat{f}^{n+1}_n(t) = 2  \int_0^1 f^{n+1}_n(x) \cos(xt)\, dx = 2 \sum_{k=0}^\infty (-1)^k \frac{t^{2k}}{(2k)!} \int_0^1 x^{2k} f_n^{n+1}(x)\, dx. \]
Now use $f_n^{n+1}(x) = d_{n+1,0} q_n(x)$ and the Mellin transform \cref{Mellinq} for $q_n$ to find
\[  \hat{f}_n^{n+1}(t) = (-1)^n 2 d_{n+1,0} \sum_{k=0}^{\infty} (-1)^k \frac{t^{2k}}{(2k)!} \frac{(-k)_n}{(2k+1)_{n+1}} . \]
Again the first $n$ terms in the series are zero, so the change $k=j+n$ gives 
\[   \hat{f}_n^{n+1}(t) = 2 d_{n+1,0} t^{2n} \sum_{j=0}^\infty (-1)^j \frac{t^{2j}}{(2j+2n)!} \frac{(-j-n)_n}{(2j+2n+1)_{n+1}} . \]
In this case we need to use
\[ (2j+2n)! (2j+2n+1)_{n+1} = (2j+3n+1)! = (3n+1)!\ 2^{2j} \bigl( \frac{3n+2}2 \bigr)_j \bigl( \frac{3n+3}2 \bigr)_j, \]
and some elementary work gives the required hypergeometric function.
\end{proof}

The Fourier transform of the other multiwavelets $f_k^{n+1}$ can be obtained by taking the sin-transform in \cref{fn+12k} or the cos-transform
in \cref{fn2k}. Both will give a determinant containing hypergeometric functions.

\section{Alpert multiwavelets and Legendre polynomials}
\label{sec:AlpertLeg}

In order to compute the multiwavelets, we introduce the notation
\[   \Psi_n(x) = \begin{pmatrix} f_1^n(2x-1) \\ f_2^n(2x-1) \\ \vdots \\ f_n(2x-1) \end{pmatrix}, \]
and the idea is to write 
\[   \Psi_n\bigl(\frac{x+1}{2}\bigr) = D_{-1}^n \Phi_n(x+1) + D_1^n \Phi_n(x)  \]
where
\[    \Phi_n(x) = \begin{pmatrix} \ell_0(x) \\ \ell_1(x) \\ \vdots \\ \ell_{n-1}(x) \end{pmatrix} \chi_{[0,1]} , \]
and $\ell_n(x)$ are the orthonormal Legendre polynomials on $[0,1]$ with $\chi_{A}$  the indicator function for the set $A$.
The scaling relation is
\[     \Phi_n(\frac{t+1}2) = C_{-1}^n \Phi_n(t+1) + C_1^n \Phi_n(t), \]
where $C_{-1}^n$ and $C_1^n$ are $n\times n$ matrices (which are lower triangular). 
The orthogonality of the Legendre polynomials gives 
\[   \int_{-1}^1  \Phi_n(\frac{x+1}2)\Phi_n^T(\frac{x+1}2)\, dx = 2 \mathbb{I}_n, \]
where $\mathbb{I}_n$ is the identity matrix of order $n$, so that  
\[   C_{-1}^n(C_{-1}^n)^T + C_1^n (C_1^n)^T = 2 \mathbb{I}_n.  \]
The orthogonality of Alpert's multiwavelets gives
\[   \int_{-1}^1 \Psi_n(\frac{x+1}2) \Psi_n^T(\frac{x+1}2)\, dx = \mathbb{I}_n  , \]
so that
\[   D_{-1}^n (D_{-1}^n)^T + D_1^n (D_1^n)^T = \mathbb{I}_n.  \] 
Furthermore, the multiwavelets are orthogonal to polynomials of degree $\leq n$, hence
\[   \int_{-1}^1 \Psi_n(\frac{x+1}2) \Phi_n^T(\frac{x+1}2)\, dx = \mathbb{O}_n, \]
so that
\[    D_{-1}^n(C_{-1}^n)^T + D_1^n (C_1^n)^T = \mathbb{O}_n.  \]

Let $n$ be even, then we write
\[   \Psi_n^e(x) = \begin{pmatrix} f_2(x) \\ f_4(x) \\ \vdots \\ f_n(x) \end{pmatrix}, \quad  
    \Psi_n^o(x) = \begin{pmatrix} f_1(x) \\ f_3(x) \\  \vdots \\ f_{n-1}(x) \end{pmatrix}.  \]
From \cref{fn+12k} and \cref{fn2k} we find that
\[   \Psi_n^e(x) = G_n^p \begin{pmatrix} p_{n/2}(x) \\ \vdots \\ p_{n-1}(x) \end{pmatrix}, \quad
     \Psi_n^o(x) = G_n^q \begin{pmatrix} q_{n/2}(x) \\ \vdots \\ q_{n-1}(x) \end{pmatrix}, \]
where $G_n^p$ and $G_n^q$ are $\frac{n}2 \times \frac{n}2$ upper triangular matrices. The orthogonality of the multiwavelets and the symmetry gives 
\[    \mathbb{I}_{n/2} = 2 \int_{0}^1 \Psi_n^e(x) \Psi_n^e(x)^T \, dx = 2G_n^p \int_{0}^1 \begin{pmatrix} p_{n/2}(x) \\ \vdots \\ p_{n-1}(x) \end{pmatrix}
     \begin{pmatrix} p_{n/2}(x) & \cdots & p_{n-1}(x) \end{pmatrix} (G_n^p)^T, \]
so that 
\[    (G_n^p)^{-1} (G_n^{pT})^{-1} = 2 \Bigl( \langle p_i,p_j \rangle \Bigr)_{i,j=n/2}^{n-1},  \]
and hence  $(G_n^p)^{-1} (G_n^{pT})^{-1}$ is the Cholesky factorization of the matrix $2 \Bigl( \langle p_i,p_j \rangle \Bigr)_{i,j=n/2}^{n-1}$.
Note however that this is an UL-factorization rather than an LU-factorization. 
In a similar way, $(G_n^q)^{-1} (G_n^{qT})^{-1}$ is the UL-Cholesky factorization of the positive definite matrix $2 \Bigl( \langle q_i,q_j \rangle \Bigr)_{i,j=n/2}^{n-1}$.
Next, we expand the $p_k$ and $q_k$ functions in terms of Legendre polynomials as
\[    \begin{pmatrix} p_0(x) \\ p_1(x) \\ \vdots \\ p_{n-1}(x) \end{pmatrix} = A_n \Phi_n(x), \quad
      \begin{pmatrix} q_0(x) \\ q_1(x) \\ \vdots \\ q_{n-1}(x) \end{pmatrix} = B_n \Phi_n(x), \]
where $A_n$ and $B_n$ are lower triangular $n \times n$ matrices. The orthonormality of the Legendre polynomials gives
\[      \int_0^1  \begin{pmatrix} p_0(x) \\ p_1(x) \\ \vdots \\ p_{n-1}(x) \end{pmatrix} 
       \begin{pmatrix} p_0(x) & p_1(x) & \cdots & p_{n-1}(x) \end{pmatrix}\, dx = A_n A_n^T, \]
so that $A_nA_n^T$ is the Cholesky factorization of the matrix $\Bigl( \langle p_i,p_j \rangle \Bigr)_{i,j=0}^{n-1}$ and similarly
$B_n B_n^T$ is the Cholesky factorization of $\Bigl( \langle q_i,q_j \rangle \Bigr)_{i,j=0}^{n-1}$.
Note that the positivity of the diagonal elements of the Cholesky factorization reflects that the leading coefficients of $p_k$ and $q_k$
are positive.
Combining this gives
\[    \Psi_n^e(x) = \Bigl( \mathbb{O}_{\frac{n}2,\frac{n}2} \ G_n^p \Bigr) A_n \Phi_n(x), \quad \Psi_n^o(x) = \Bigl( \mathbb{O}_{\frac{n}2,\frac{n}2} 
     \ G_n^q \Bigr) B_n \Phi_n(x), \]
where $\mathbb{O}_{n,m}$ is a $n\times m$ matrix containing only zeros,
and $D_1^n$ is obtained by intertwining $\Psi_n^e$ and $\Psi_n^o$, i.e.,
\[     D_1^n = \Bigl( \mathbb{O}_{n,\frac{n}2} \ \hat{G}_n^p \Bigr) A_n  + \Bigl( \mathbb{O}_{n,\frac{n}2} \ \hat{G}_n^q \Bigr) B_n, \]
where $\hat{G}_n^p$ is a stretched version of $G_n^p$ where zero rows have been inserted between every row, starting from the first row,
and $\hat{G}_n^q$ is obtained from $G_n^q$ by inserting zero rows starting from the second row:
\[    \hat{G}_n^p = \begin{pmatrix} 0 & 0 & 0 & \cdots & 0 \\ 1 & 0 & 0 & \cdots & 0 \\ 0 & 0 & 0 & \cdots & 0 \\ 0 & 1 & 0 & \cdots & 0 \\
               \vdots & & & \vdots & \vdots \\ 0 & 0 & 0 & \cdots & 0 \\ 0 & 0 & 0 & \cdots & 1 \end{pmatrix}  G_n^p,  \quad
       \hat{G}_n^q = \begin{pmatrix} 1 & 0 & 0 & \cdots & 0 \\ 0 & 0 & 0 & \cdots & 0 \\ 0 & 1 & 0 & \cdots & 0 \\ 0 & 0 & 0 & \cdots & 0 \\
               \vdots & & & \vdots & \vdots \\ 0 & 0 & 0 & \cdots & 1 \\ 0 & 0 & 0 & \cdots & 0 \end{pmatrix}  G_n^q.  \]

For $n$ odd we have to modify the method only slightly. We now use
\[   \Psi_n^e(x) = \begin{pmatrix} f_2(x) \\ f_4(x) \\ \vdots \\ f_{n-1}(x) \end{pmatrix}, \quad  
     \Psi_n^o(x) = \begin{pmatrix} f_1(x) \\ f_3(x) \\  \vdots \\ f_{n}(x) \end{pmatrix},  \]
so that $\Psi_n^e$ has $(n-1)/2$ elements and $\Psi_n^o$ has $(n+1)/2$ elements. From \cref{fn+12k} and \cref{fn2k} we now find
\[   \Psi_n^e(x) = G_n^q \begin{pmatrix} q_{(n+1)/2}(x) \\ \vdots \\ q_{n-1}(x) \end{pmatrix}, \quad
     \Psi_n^o(x) = G_n^p \begin{pmatrix} p_{(n-1)/2}(x) \\ \vdots \\ p_{n-1}(x) \end{pmatrix}, \]
where $G_n^p$ is a $\frac{n+1}2 \times \frac{n+1}2$ and $G_n^q$ a $\frac{n-1}2 \times \frac{n-1}2$ upper triangular matrix.
As before,  $(G_n^p)^{-1} (G_n^{pT})^{-1}$ is the UL-Cholesky factorization of the matrix $2 \Bigl( \langle p_i,p_j \rangle \Bigr)_{i,j=(n-1)/2}^{n-1}$ 
and $(G_n^q)^{-1} (G_n^{qT})^{-1}$ is the UL-Cholesky factorization of the positive definite matrix $2 \Bigl( \langle q_i,q_j \rangle \Bigr)_{i,j=(n+1)/2}^{n-1}$. The lower triangular matrices $A_n$ and $B_n$ are the same as before. The final result is that
\[     D_1^n = \Bigl( \mathbb{O}_{n,\frac{n-1}2} \ \hat{G}_n^p \Bigr) A_n  + \Bigl( \mathbb{O}_{n,\frac{n+1}2} \ \hat{G}_n^q \Bigr) B_n, \]
where the stretched matrices are now given by
\[    \hat{G}_n^p = \begin{pmatrix} 1 & 0 & 0 & \cdots & 0  \\ 0 & 0 & 0 & \cdots & 0 \\ 0 & 1 & 0 & \cdots & 0 \\ 0 & 0 & 0 & \cdots &  \\
               \vdots & & & \vdots & \vdots \\ 0 & 0 & 0 & \cdots & 0 \\ 0 & 0 & 0 & \cdots & 1 \end{pmatrix}  G_n^p,  \quad
       \hat{G}_n^q = \begin{pmatrix} 0 & 0 & 0 & \cdots & 0 \\ 1 & 0 & 0 & \cdots & 0 \\ 0 & 0 & 0 & \cdots & 0 \\ 0 & 1 & 0 & \cdots & 0 \\ 
               \vdots & & & \vdots & \vdots \\ 0 & 0 & 0 & \cdots & 1 \\ 0 & 0 & 0 & \cdots & 0 \end{pmatrix}  G_n^q.  \]

\appendix
\section{An alternative proof of \cref{prop:ppqq}}

The Mellin transforms \cref{Mellinp}--\cref{Mellinq} will be useful, if we combine this with Parseval's formula for the Mellin transform
\begin{equation}   \label{parseval}
   \int_0^\infty f(x)g(x)\, dx = \frac{1}{2\pi} \int_{-\infty}^\infty \hat{f}(-\frac12 +it) \overline{\hat{g}(-\frac12 + it)}\, dt, 
\end{equation}
where 
\[   \hat{f}(s) = \int_0^\infty f(x)x^s\, dx, \quad  \hat{g}(s) = \int_0^\infty g(x) x^s\, dx.  \]
By using Parseval's formula and \cref{Mellinp} we have
\[  \int_0^1 p_n(x)p_k(x)\, dx = (-1)^{n+k} \frac{1}{2\pi} \int_{-\infty}^\infty 
   \frac{\Gamma(\frac34+n - \frac{it}2) \Gamma(\frac34 + k+ \frac{it}2) |\Gamma(\frac12 +it)|^2}
        {\Gamma(\frac32+n +it) \Gamma(\frac32+k-it) |\Gamma(\frac34 +\frac{it}{2})|^2} \, dt.  \]
The arguments $it/2$ are unusual, so to get rid of them we change the variable $t=2s$ to find
\[  \langle p_n,p_k \rangle = (-1)^{n+k} \frac{1}{\pi} \int_{-\infty}^\infty
   \frac{\Gamma(\frac34+n - is) \Gamma(\frac34 + k+ is) |\Gamma(\frac12 +2is)|^2}
        {\Gamma(\frac32+n +2is) \Gamma(\frac32+k-2is) |\Gamma(\frac34 +is)|^2} \, ds.  \]
Now we use Legendre's duplication formula for the Gamma function
\[     \Gamma(2z) = \frac{2^{2z-1}}{\sqrt{\pi}} \Gamma(z)\Gamma(z+\frac12)  \]
to find
\begin{multline*}
  \langle p_n,p_k \rangle = (-1)^{n+k} 2^{-(n+k+1)}  \\
   \times \frac{1}{\pi} \int_{-\infty}^\infty
    \frac{\Gamma(\frac34+n-is) \Gamma(\frac34+k+is) |\Gamma(\frac14+is)|^2}
         {\Gamma(\frac{n}2+\frac34+is) \Gamma(\frac{n}2+\frac54+is) \Gamma(\frac{k}2 +\frac34 -is)\Gamma(\frac{k}2+\frac54-is)}\, ds . 
\end{multline*}
If we now change the variable $is$ to $y$, then we get a Mellin-Barnes integral defining a Meijer G-function
\begin{equation}   \label{pnG}
   \int_0^1 p_n(x)p_k(x)\, dx = (-1)^{n+k} 2^{-(n+k+1)} G_{4,4}^{2,2}\left( 1; \begin{array}{c}
           -k+\frac14, \frac34, \frac{k}{2}+\frac34, \frac{k}{2}+\frac54 \\
            n+\frac34, \frac14 , -\frac{n}{2}+\frac14, -\frac{n}2-\frac14  \end{array}  \right), 
\end{equation}
where the Meijer G-function is defined as
\[    G_{p,q}^{m,n}\left( z; \begin{array}{c} a_1,\ldots,a_p \\ b_1, \ldots,b_q \end{array} \right)
    = \frac{1}{2\pi i} \int_\Gamma \frac{ \prod_{j=1}^m \Gamma(b_j-s)  \prod_{j=1}^n \Gamma(1-a_j+s)}
                                        {\prod_{j=m+1}^q \Gamma(1-b_j+s) \prod_{j=n+1}^p \Gamma(a_j-s)} z^s\, ds,  \] 
where the path $\Gamma$ separates the poles of $\Gamma(b_j-s)$ from the poles of $\Gamma(1-a_j+s)$ \cite[\S 16.17]{DLMF,OLBC10}.
In our case we have $m=n=2$, $p=q=4$ and
\[    b_1=n+\frac34, \ b_2=\frac14, \ b_3 = -\frac{n}{2}+\frac14, \ b_4 = -\frac{n}2-\frac14 , \]
\[    a_1 = -k+\frac14, \ a_2 = \frac34, \ a_3 = \frac{k}2 + \frac34, \ a_4 = \frac{k}2 + \frac54.  \]
This Meijer G-function can be expressed as a linear combination of two hypergeometric functions
{\small
\begin{align*}
   G_{4,4}^{2,2}\left( z; \begin{array}{c} a_1,a_2,a3,a_4 \\ b_1,b_2,b_3,b_4 \end{array} \right)  
  =\  &\frac{z^{b_1}\Gamma(b_2-b_1)\Gamma(1+b_1-a_1)\Gamma(1+b_1-a_2)}{\Gamma(1+b_1-b_3)\Gamma(1+b_1-b_4)\Gamma(a_3-b_1)\Gamma(a_4-b_1)} \\
     &\times {}_4F_3\left( \begin{array}{c} 1+b_1-a_1,1+b_1-a_2,1+b_1-a_3,1+b_1-a_4 \\ 1+b_1-b_2,1+b_1-b_3,1+b_1-b_4 \end{array} ; z \right) \\
     &+ \ \frac{z^{b_2} \Gamma(b_1-b_2)\Gamma(1+b_2-a_1)\Gamma(1+b_2-a_2)}{\Gamma(1+b_2-b_3)\Gamma(1+b_2-b_4)\Gamma(a_3-b_2)\Gamma(a_4-b_2)} \\
     &\times {}_4F_3\left( \begin{array}{c} 1+b_2-a_1,1+b_2-a_2,1+b_2-a_3,1+b_2-a_4 \\ 1+b_2-b_1,1+b_2-b_3,1+b_2-b_4 \end{array} ; z \right).
\end{align*}}
We need the case $z=1$. Note that the ${}_4F_3$ hypergeometric functions are balanced \cite[\S 16.4 (i)]{DLMF,OLBC10} : for the first ${}_4F_3$
the sum of the parameters in the numerator is $4n+4$ and the sum of the denominator parameters is $4n+5$, for the second ${}_4F_3$ the
numerator parameters add up to 2 and the denominator parameters to 3. Observe that $a_3-b_1=\frac{k}2-n$ and $a_4-b_1=\frac{k+1}2-n$, hence one of these two is a negative integer, and the Gamma function $\Gamma(a_3-b_1)$ or $\Gamma(a_4-b)$ has a pole, so that the first term vanishes. This means that \cref{pnkF} follows.

In a similar way one can also obtain the following expression for the integrals involving the $q_n$ polynomials
\begin{equation}   \label{qnG}
       \int_0^1 q_n(x)q_k(x)\, dx = (-1)^{n+k} 2^{-(n+k+1)} G_{4,4}^{2,2}\left( 1; \begin{array}{c}
           -k+\frac34, \frac14, \frac{k}{2}+\frac34, \frac{k}{2}+\frac54 \\
            n+\frac14, \frac34 , -\frac{n}{2}+\frac14, -\frac{n}2-\frac14  \end{array}  \right). 
\end{equation}
which simplifies to \cref{qnkF}.

\section*{Acknowledgments}
The research of the first author was partially supported by Simons Foundation Grant 210169.
The research of the second author was partially supported by Simons Foundation Grant 280940.
The third author was supported by FWO research projects G.0934.13 and G.0864.16 and KU Leuven research grant OT/12/073.
This work was done while the third author was visiting Georgia Institute of Technology and he would like to thank
FWO-Flanders for the financial support of his sabbatical and the School of Mathematics at GaTech for their hospitality.

\bibliographystyle{siamplain}
\bibliography{multiwavelet-refs}
\end{document}

\maketitle

\section{Explicit matrices for the scaling relation and the multiwavelets}

The Legendre polynomials are denoted by $P_n(x)$ and they are orthogonal on $[-1,1]$
\[   \int_{-1}^1 P_n(x)P_m(x)\, dx = \frac{2}{2n+1} \delta_{n,m} . \]
We will be using the orthonormal Legendre polynomials on $[0,1]$ given by
\[    \ell_n(x) = \sqrt{2n+1} P_n(2x-1).  \]
We denote the scaling vector by
\[    \Phi_n(x) = \begin{pmatrix} \ell_0(x) \\ \ell_1(x) \\ \vdots \\ \ell_{n-1}(x) \end{pmatrix} \chi_{[0,1]} \]
and the vector containing the multiwavelets by
\[   \Psi(t) = \begin{pmatrix} f_1^n(2t-1) \\ f_2^n(2t-1) \\ \vdots \\ f_n^n(2t-1)  \end{pmatrix}  \chi_{[0,1]} , \]
where $\chi_{A}$ is the indicator function for the set $A$.
The scaling relation is
\begin{equation}  \label{PhiC}
     \Phi_n(\frac{t}2) = C_{-1}^n \Phi_n(t) + C_1^n \Phi_n(t-1), 
\end{equation}
where $C_{-1}^n$ and $C_1^n$ are $n\times n$ matrices (which are lower triangular). 
The wavelets are expressed in terms of Legendre polynomials by
\begin{equation}  \label{PsiD}
   \Psi_n(\frac{t+1}2) = D_{-1}^n \Phi_n(t+1) + D_1^n \Phi_n(t), 
\end{equation}
where $D_{-1}^n$ and $D_1^n$ are $n\times n$ matrices. 
The orthogonality of the Legendre polynomials gives 
\[   \int_{-1}^1  \Phi_n(\frac{x+1}2) \Phi_n^T(\frac{x+1}2)\, dx = 2 \mathbb{I}_n, \]
where $\mathbb{I}_n$ is the identity matrix of order $n$, so that  
\begin{equation}  \label{CC}
   C_{-1}^n(C_{-1}^n)^T + C_1^n (C_1^n)^T = 2 \mathbb{I}_n.  
\end{equation}
The orthogonality of Alpert's multiwavelets gives
\[   \int_{-1}^1 \Psi_n(\frac{x+1}2) \Psi_n^T(\frac{x+1}2)\, dx = \mathbb{I}_n  , \]
so that
\[   D_{-1}^n (D_{-1}^n)^T + D_1^n (D_1^n)^T = \mathbb{I}_n.  \] 
Furthermore, the multiwavelets in $\Psi_n$ are orthogonal to polynomials of degree $\leq n$, hence
\[   \int_{-1}^1 \Psi_n(\frac{x+1}2) \Phi_n^T(\frac{x+1}2)\, dx = \mathbb{O}_n, \]
so that
\[    D_{-1}^n(C_{-1}^n)^T + D_1^n (C_1^n)^T = \mathbb{O}_n.  \]

\subsection{The scaling relation and the matrices $C_{-1}^n$}
The matrices $C_1^n$ for $1 \leq n \leq 4$ are given explicitly in \cite[p.~2486]{GerMar} (but their notation is a bit different for the multiplicity index $n$). Here is a simple way to generate them. The orthonormal Legendre polynomials
on $[0,1]$ satisfy the three term recurrence relation
\[    t\ell_{n-1}(x) = a_{n} \ell_{n}(t) + \frac12 \ell_{n-1}(t) + a_{n-1} \ell_{n-2}(t), \]
where 
\[    a_n = \frac{n}{2\sqrt{(2n-1)(2n+1)}} .  \]
Introducing the Jacobi matrix
\begin{equation}   \label{Jacobi}
    J_{n} = \begin{pmatrix} 1/2 & a_1 & 0 & 0 & \cdots & 0 \\
                                a_1 & 1/2 & a_2 & 0 & \cdots & 0 \\
                                 0 & a_2 & 1/2 & a_3 &  &  \vdots \\
                                 \vdots &  & \ddots & \ddots & \ddots & 0 \\
                                 0 & \cdots & 0 & a_{n-2} & 1/2 & a_{n-1} \\
                                 0 & \cdots  & 0 & 0 & a_{n-1} & 1/2 
                \end{pmatrix}
\end{equation}
then gives
\begin{equation}   \label{JPhi}
    J_{n}\Phi_{n}(t) = t \Phi_{n}(t) - a_{n}\ell_{n}(t) e_{n} \chi_{[0,1]}(t),
\end{equation}
where $e_n$ is the unit vector $(0,\ldots,0,1)^T$ in $\mathbb{R}^n$. If we use \cref{JPhi} in the scaling relation
\cref{PhiC}, then
\[     J_n\Phi_n(t/2) = J_n C_{-1}^n \Phi_n(t) + J_n C_1^n\Phi_n(t-1).  \]
Note that $J_n$ is the truncated Jacobi matrix for the Legendre polynomials on $[0,1]$ and that its eigenvalues are the zeros
of $p_n$, which are all in $(0,1)$, hence $J_n$ is not singular (and positive definite), so that its inverse exists. This allows
to write
\[    J_n \Phi_n(t/2) = J_nC_{-1}^n J_n^{-1} J_n \Phi_n(t) + J_n C_1^n J_n^{-1} J_n \Phi_n(t-1).  \]
If we use \cref{JPhi} three times, then
\begin{multline}  \label{JCJ}
   \frac{t}{2} \Phi_n(\frac{t}{2}) - a_n \ell_n(\frac{t}{2}) \chi_{[0,2]}e_n = J_n C_{-1}^n J_n^{-1}
     \Bigl( t\Phi_n(t) - a_n\ell_n(t) \chi_{[0,1]} e_n \Bigr) \\
        +\ J_nC_1^nJ_n^{-1} \Bigl( (t-1) \Phi_n(t-1) - a_n\ell_n(t-1) \chi_{[1,2]} e_n \Bigr).  
\end{multline}
Taking $t=0$ gives
\[   -a_n\ell_n(0) e_n = J_nC_{-1}^n J_n^{-1} \Bigl( -a_n \ell_n(0) e_n \Bigr), \]
and since $\ell_n(0) \neq 0$ we thus find 
\begin{equation}   \label{JCJen}
J_nC_{-1}^nJ_n^{-1} e_n = e_n .
\end{equation}
Using \cref{PhiC} on the left side of \cref{JCJ} and \cref{JCJen} on the right side gives for $t \in [0,1]$
\begin{equation*}  \label{JCJ01}  
  \frac{t}{2} C_{-1}^n \Phi_n(t) = a_n \ell_n(t/2) e_n - a_n \ell_n(t) e_n
     + tJ_nC_{-1}^n J_n^{-1} \Phi_n(t). 
\end{equation*} 
Divide this by $t$, multiply on the right by $\Phi_n^T(t)$ and integrate over $[0,1]$, to find
\[    C_{-1}^n = 2J_nC_{-1}^n J_n^{-1} + 2 a_n \int_0^1 \frac{\ell_n(t/2)-\ell_n(t)}{t} e_n \Phi_n^T(t)\, dt, \]
where we used the orthonormality of the Legendre polynomials
\[   \int_0^1  \Phi_n(t)\Phi_n^T(t)\, dt = \mathbb{I}_n.  \]
Let us introduce the column vector $m_n$, with
\[   (m_n)_j = 2a_n \int_0^1 \frac{\ell_n(t/2)-\ell_n(t)}{t} p_j(t)\, dt , \qquad 0 \leq j \leq n-1.  \]
and $M_n = e_nm_n^T$ as the matrix with zeros everywhere, except for the last row which is $m_n^T$, then we have the matrix relation
\[     C_{-1}^n = 2J_nC_{-1}^nJ_n^{-1} + M_n, \]
or, after multiplication the the right by $J_n$,
\begin{equation}  \label{CJ-JC}   
C_{-1}^n J_n = 2J_n C_{-1}^n + M_n J_n.  
\end{equation}  
This is the matrix version of the scaling relation \cref{PhiC}, and the $2$ results from the fact that our scaling is with a factor $2$.

We now describe a recursive way to compute the matrices $C_{-1}^n$. Write
\begin{equation}   \label{Cn+1}
   C_{-1}^{n+1} = \begin{pmatrix} C_{-1}^n & 0_n \\ r_n^T & s_{n+1} \end{pmatrix}, 
\end{equation}
where $0_n$ and $r_n$ are vectors of size $n$ and $s_{n+1}$ is a real number. One also has
\[   J_{n+1} = \begin{pmatrix} J_n & a_n e_n \\ a_n e_n^T & 1/2   \end{pmatrix}.  \]
Then, ignoring the last row, we have
\[   C_{-1}^{n+1} J_{n+1} = \begin{pmatrix}  C_{-1}^n J_n  & a_n s_n e_n \\ \ast & \ast  \end{pmatrix}, \]
and
\[   J_{n+1} C_{-1}^{n+1} = \begin{pmatrix} J_n C_{-1}^n + a_ne_nr_n^T & a_n s_{n+1}e_n \\ \ast & \ast \end{pmatrix}.  \]
From \cref{CJ-JC} for $n+1$ we find
\[     C_{-1}^{n+1} J_{n+1} - 2 J_{n+1} C_{-1}^{n+1} = M_{n+1}J_{n+1}.  \]
Ignoring the last row, this gives
\begin{equation}   \label{CJ-Jc-r}
   C_{-1}^n J_n - 2 J_n C_{-1}^n - 2a_n e_n r_n^T = 0, 
\end{equation}
and $a_ns_ne_n - 2a_ns_{n+1}e_n=0$. The latter gives $2s_{n+1}=s_n$ and since $C_{-1}^1 = 1$ (and thus $s_1=1$) we
find immediately that $s_{n+1} =1/2^n$ (which was already given in \cite{GerMar,GerIliev}). So recursively we start with 
$C_{-1}^1=1$. If $C_{-1}^n$ is known, then $C_{-1}^{n+1}$ is given by \cref{Cn+1}, where the last row $(r_n^T, s_{n+1})$
contains $s_{n+1}=1/2^n$ and the vector $r_n$ which is computed as the last row in
\[    2a_n e_n r_n^T = C_{-1}^nJ_n - 2J_n C_{-1}^n.  \]
The matrices $C_1^n$ can be computed from $C_{-1}^n$ by using $(C_{-1}^n)_{i,j} = (-1)^{i+j} (C_1^n)_{i,j}$.
The latter relation, combined with \cref{CC}, gives
\[    2 \delta_{i_j} = (1+(-1)^{i+j}) \sum_{k=1}^n (C_{-1}^n)_{i,k} (C_{-1}^n)_{j,k} , \]
so that the even rows of $C_{-1}^n$ are orthonormal vectors and the odd rows of $C_{-1}^n$ are orthonormal. This is a useful
check.
 
The matrix $C_{-1}^{10}$ is given below. The matrices for $n < 10$ are found as the principal $n\times n$ submatrix of this matrix.

\[ \small 
     \kern-10pt \begin{pmatrix}
        1 & 0 & 0 & 0 & 0 & 0 & 0 & 0 & 0 & 0 \\
        -\frac{\sqrt{3}}{2} & \frac{1}{2} & 0 & 0 & 0 & 0 & 0 & 0 & 0 & 0 \\[6pt]
        0 & -\frac{\sqrt{15}}{4} & \frac{1}{4} & 0 & 0 & 0 & 0 & 0 & 0 & 0 \\[6pt]
        \frac{\sqrt{7}}{8}  & \frac{\sqrt{21}}{8} & -\frac{\sqrt{35}}{8} & \frac{1}{8} & 0 & 0 & 0 & 0 & 0 & 0 \\[6pt]
        0 & \frac{\sqrt{3}}{8}& \frac{3\sqrt{5}}{8} & -\frac{3\sqrt{7}}{16} & \frac{1}{16} & 0 & 0 & 0 & 0 & 0 \\[6pt]
        -\frac{\sqrt{11}}{16} & -\frac{\sqrt{33}}{16} & -\frac{\sqrt{55}}{32} & \frac{3\sqrt{77}}{32} & -\frac{3\sqrt{11}}{32} & \frac{1}{32} 
        & 0 & 0 & 0 & 0 \\[6pt]
        0 & -\frac{\sqrt{39}}{64} & -\frac{3\sqrt{65}}{64} & -\frac{\sqrt{91}}{16} & \frac{3\sqrt{13}}{16} & -\frac{\sqrt{143}}{64} 
        & \frac{1}{64} & 0 & 0 & 0 \\[6pt]
        \frac{5\sqrt{15}}{128} & \frac{15\sqrt{5}}{128} & \frac{19\sqrt{3}}{128} & -\frac{\sqrt{105}}{128} & -\frac{25\sqrt{15}}{128}  
        & \frac{5\sqrt{165}}{128} & -\frac{\sqrt{195}}{128} & \frac{1}{128} & 0 & 0 \\[6pt] 
        0 & \frac{\sqrt{51}}{128} & \frac{3\sqrt{85}}{128} & \frac{5\sqrt{119}}{128} & \frac{9\sqrt{17}}{128} & -\frac{7\sqrt{187}}{128} 
        & \frac{3\sqrt{221}}{128} & -\frac{\sqrt{255}}{256} & \frac{1}{256} & 0 \\[6pt]
        -\frac{7\sqrt{19}}{256} & -\frac{7\sqrt{57}}{256} & -\frac{3\sqrt{95}}{128} & -\frac{\sqrt{133}}{128} & \frac{9\sqrt{19}}{128}  
        & \frac{5\sqrt{209}}{128} & -\frac{21\sqrt{247}}{512} & \frac{7\sqrt{285}}{512} & -\frac{\sqrt{323}}{512} & \frac{1}{512}
      \end{pmatrix}  \]

\subsection{Explicit expressions for $D_1^n$ for small $n$}
The multiwavelets for multiplicities $1 \leq n \leq 5$ are given in \cite[Table 1 on p.~259]{Alpert}. 
We have computed the matrices $D_1^n$ and the matrices $D_{-1}^n$ can be found using
\[     (D_{-1}^n)_{i,j} = (-1)^{i+j+n-1} D_1^n.  \]
All the matrices are of the form
\[    D_1^n =   \hat{D}_1^n \diag(1,\frac{1}{\sqrt{3}},\frac{1}{\sqrt5},\ldots,\frac{1}{\sqrt{2n-1}}) , \]
where the diagonal matrix on the right contains the normalizing constants for the Legendre polynomials. This means that
the multiwavelets of multiplicity $n$ are given in terms of Legendre polynomials by
\[   \begin{pmatrix} f_1^n(x) \\ f_2^n(x) \\ \vdots \\ f_n^n(x) \end{pmatrix}   = \hat{D}_1^n
     \begin{pmatrix} P_0(2x-1) \\ P_1(2x-1) \\ \vdots \\ P_{n-1}(2x-1)  \end{pmatrix} , \qquad x \in [0,1].  \]  

The fist ten matrices $\hat{D}_1^n$ are given by
\[   \hat{D}_1^1 = \diag \begin{pmatrix} \frac{\sqrt{2}}{2} \end{pmatrix}  \begin{pmatrix}  1  \end{pmatrix}, \]

\[   \hat{D}_1^2  = \diag \begin{pmatrix} \frac{\sqrt{6}}2 & \frac{3\sqrt{2}}{4} \end{pmatrix} 
      \begin{pmatrix}  0 & 1 \\ -\frac13 & 1  \end{pmatrix},   \]
            
\[   \hat{D}_1^3 = \diag \begin{pmatrix}    
          \frac{5\sqrt{2}}{6} & \frac{5\sqrt{6}}{8} & \frac{\sqrt{10}}{3} \end{pmatrix}
   \begin{pmatrix} -\frac{1}5 & \frac35 & 1 \\[6pt] 0 & -\frac15 & 1 \\[8pt] \frac14 & -\frac34 & 1 \end{pmatrix}, \]
              
\[   \hat{D}_1^4 = 
    \diag \begin{pmatrix} \frac{7\sqrt{510}}{170} & \frac{\sqrt{42}}{4} & \frac{4\sqrt{1190}}{85} & \frac{\sqrt{210}}{16} \end{pmatrix}
        \begin{pmatrix} 
          0 & -\frac27 & \frac{10}7 & 1 \\[6pt] 
          \frac{2}{21} & -\frac27 & \frac5{21} & 1 \\[6pt] 
          0 & \frac{3}{32} & -\frac{15}{32} & 1 \\[6pt]
          -\frac5{21} & \frac{5}{7} & -\frac{23}{21} & 1  
         \end{pmatrix} , \]

\[   \hat{D}_1^5 = 
     \diag \begin{pmatrix} 
         \frac{3 \sqrt{186}}{62} & \frac{9 \sqrt{38}}{38} & \frac{45 \sqrt{514290}}{17143} & \frac{33 \sqrt{798}}{608} & \frac{12 \sqrt{1106}}{553}
      \end{pmatrix}  \]
\[             \begin{pmatrix} 
             2 & -6 & 5 & 21 & 9 \\[6pt]
                             0 & \frac32 & -\frac{15}2 & 14 & 18 \\[6pt]
                             -\frac52 & \frac{15}2 & -\frac{67}{7} & -3 & \frac{270}{7} \\[6pt]
                             0 & -1 & 5 & -\frac{25}{2} & \frac{33}{2} \\[6pt]
                             \frac74 & -\frac{21}4 & \frac{235}{28} & -\frac{39}{4} & \frac{48}{7} 
\end{pmatrix}, \]
\newpage 

\[  \hat{D}_1^6 =
   \diag \begin{pmatrix} \frac {11\sqrt{70}}{210} & \frac {55\sqrt{13490}}{5396} & \frac {506\sqrt{32270}}{48405} &
   \frac {845\sqrt{89034}}{129504} & \frac {4\sqrt{212982}}{1383} & \frac {13\sqrt{66}}{192}  \end{pmatrix}  \]
\[  \begin{pmatrix}   
      0 & \frac{3}{11} & -\frac {15}{11} & \frac{28}{11} & \frac{36}{11} & 1 \\[6pt] 
     -\frac{1}{11} & \frac{3}{11} & -\frac{19}{55} & -\frac{7}{55} & \frac{15}{11} & 1 \\[6pt] 
      0 & -\frac{39}{1012} & \frac{195}{1012} & -\frac{3647}{8096} & \frac{261}{736} & 1 \\[6pt] 
      \frac{504}{9295} & -\frac{1512}{9295} & \frac{437}{1859} & -\frac{7}{55} & -\frac{3513}{9295} & 1 \\[6pt]
      0 & \frac{3}{64} & -\frac{15}{64} & \frac{79}{128} & -\frac{135}{128} & 1\\[6pt]
      -\frac {42}{143} & \frac{126}{143} & -\frac{205}{143} & \frac{259}{143} & -\frac{249}{143} & 1 
      \end{pmatrix} ,  \] 

\begin{multline*}
   \tilde{D}_1^7 =
     \diag \Bigl( \begin{matrix}
             \frac{13 \sqrt{54610}}{10922} & \frac{13 \sqrt{30}}{80} & \frac{21931 \sqrt{26922730}}{67306825}  \end{matrix} \\
           \begin{matrix}    
             \frac{299 \sqrt{34230}}{26080} & \frac{129584 \sqrt{790619170}}{1976547925} & \frac{221 \sqrt{10758}}{20864} &
             \frac{32 \sqrt{4169594}}{160369} \end{matrix} \Bigr) 
\end{multline*} 
\[     \begin{pmatrix}
	-\frac{5}{13} & \frac{15}{13} & -\frac{19}{13} & -\frac{7}{13} & \frac{75}{13} & \frac{55}{13} & 1 \\[6pt] 
	0 & -\frac{1}{13} & \frac{5}{13} & -\frac{35}{39} & \frac{9}{13} & \frac{77}{39} & 1 \\[6pt] 
	\frac{641}{12532} & -\frac{1923}{12532} & \frac{19235}{87724} & -\frac{99}{964} & -\frac{34539}{87724} & \frac{5445}{6748} & 1 \\[6pt] 
	0 & \frac{7}{299} & -\frac{35}{299} & \frac{265}{897} & -\frac{123}{299} & \frac{11}{897} & 1 \\[6pt] 
	-\frac{15279}{296192} & \frac{45837}{296192} & -\frac{490215}{2073344} & \frac{62403}{296192} & \frac{40629}{518336} & 
        -\frac{26895}{39872} & 1 \\[6pt]
	0 & -\frac{9}{221} & \frac{45}{221} & -\frac{28}{51} & \frac{228}{221} & -\frac{899}{663} & 1 \\[6pt] 
	\frac{363}{1024} & -\frac{1089}{1024} & \frac{3575}{2048} & -\frac{4697}{2048} & \frac{5091}{2048} & -\frac{4213}{2048} & 1
	\end{pmatrix} \]

\begin{multline*}
  \tilde{D}_1^8 =
    \diag \Bigl( \begin{matrix}
             \frac{\sqrt{7710}}{514} & \frac{21\sqrt{6890}}{2756} & \frac{744\sqrt{12310934790}}{58623499} & 
             \frac{266441\sqrt{6383647010}}{10213835216} \end{matrix}  \\
    \begin{matrix} \frac{51536\sqrt{56134851630}}{5613485163} &  
             \frac{1729665\sqrt{530024297594}}{770944432864} & \frac{320\sqrt{541398}}{270699} & \frac{323\sqrt{37184290}}{6656768} 
           \end{matrix}  \Bigr) 
\end{multline*}
\[      \begin{pmatrix} 
        0 & -\frac{2}{5} & 2 & -\frac{14}{3} & \frac{18}{5} & \frac{154}{15} & \frac{26}{5} & 1 \\[6pt] 
	\frac{2}{15} & -\frac{2}{5} & \frac{4}{7} & -\frac{4}{15} & -\frac{36}{35} & \frac{44}{21} & \frac{13}{5} & 1 \\[6pt] 
	0 & \frac{473}{14880} & -\frac{473}{2976} & \frac{223}{558} & -\frac{169}{310} & -\frac{2167}{89280} & \frac{37817}{29760} & 1 \\[6pt]
	-\frac{6778}{190315} & \frac{20334}{190315} & -\frac{7582}{47019} & \frac{75118}{570945} & \frac{128307}{1332205} &
	-\frac{394295}{799323} & \frac{230477}{570945} & 1 \\[6pt]
	0 & -\frac{7007}{412288} & \frac{35035}{412288} & -\frac{367003}{1649152} & \frac{612339}{1649152} & -\frac{247467}{824576} &
	-\frac{246753}{824576} & 1 \\[6pt]
	\frac{30778}{576555} & -\frac{30778}{192185} & \frac{174373}{691866} & -\frac{133193}{494190} & \frac{1212121}{12684210} &
	\frac{1289921}{3459330} & -\frac{972907}{1001385} & 1 \\[6pt] 
	0 & \frac{1573}{40960} & -\frac{1573}{8192} & \frac{5369}{10240} & -\frac{10569}{10240} & \frac{63261}{40960} & -\frac{67977}{40960} & 1 \\[6pt] 
	-\frac{143}{323} & \frac{429}{323} & -\frac{2123}{969} & \frac{2849}{969} & -\frac{12021}{3553} & \frac{3151}{969} & -\frac{25273}{10659} & 1   
      \end{pmatrix}  \] 
\newpage

\rotatebox{90}
{\parbox{6.8in}{\tiny
\begin{multline*}  \tilde{D}_1^9 = \diag 
   \Bigl(  \begin{matrix} 
            \frac{17 \sqrt{3066}}{9198} & \frac{17 \sqrt{305718}}{21837} & \frac{254099 \sqrt{195086514}}{3218927481} &
            \frac{15181 \sqrt{1350764030}}{294712152}  \end{matrix} \\
           \begin{matrix} 
            \frac{27011980 \sqrt{4525963475306}}{24892799114183} & \frac{25517 \sqrt{21593481910}}{1812040440} & 
            \frac{488648 \sqrt{9441081992706}}{1089355614543} & \frac{7429 \sqrt{38399790}}{68743680} & \frac{64 \sqrt{293331090}}{5176431}
           \end{matrix}  \Bigr)
\end{multline*}
\[     \begin{pmatrix}
	\frac{14}{17} & -\frac{42}{17} & \frac{60}{17} & -\frac{28}{17} & -\frac{108}{17} & \frac{220}{17} & \frac{273}{17} &
	\frac{105}{17} & 1 \\[6pt] 
	0 & \frac{7}{68} & -\frac{35}{68} & \frac{22}{17} & -\frac{30}{17} &- \frac{11}{136} & \frac{559}{136} & \frac{55}{17} & 1 \\[6pt] 
	-\frac{16093}{254099} & \frac{48279}{254099} & -\frac{8565}{29894} & \frac{118027}{508198} & \frac{179253}{1016396} &
	-\frac{893519}{1016396} & \frac{174447}{254099} & \frac{26040}{14947} & 1 \\[6pt] 
	0 & -\frac{257}{15181} & \frac{1285}{15181} & -\frac{147357}{667964} & \frac{240645}{667964} & -\frac{7915}{30362} &
	-\frac{117611}{333982} & \frac{132641}{166991} & 1 \\[6pt] 
	\frac{3125551}{108047920} & -\frac{9376653}{108047920} & \frac{23338523}{172876672} & -\frac{116724881}{864383360} &
	\frac{12646833}{864383360} & \frac{228799241}{864383360} & -\frac{53297647}{108047920} & \frac{83601}{1271152} & 1 \\[6pt] 
	0 & \frac{5687}{408272} & -\frac{28435}{408272} & \frac{209097}{1122748} & -\frac{380505}{1122748} & \frac{2103635}{5307536} &
	-\frac{498977}{4490992} & -\frac{2187635}{3648931} & 1 \\[6pt] 
	-\frac{915629}{15636736} & \frac{2746887}{15636736} & -\frac{8797695}{31273472} & \frac{10308641}{31273472} &
	-\frac{7081371}{31273472} & -\frac{4364723}{31273472} & \frac{5933817}{7818368} & -\frac{584115}{459904} & 1 \\[6pt] 
	0 & -\frac{286}{7429} & \frac{1430}{7429} & -\frac{3934}{7429} & \frac{7950}{7429} & -\frac{165350}{96577} & \frac{16006}{7429} &
	-\frac{189745}{96577} & 1 \\[6pt] 
	\frac{9295}{16384} & -\frac{27885}{16384} & \frac{46137}{16384} & -\frac{62699}{16384} & \frac{74655}{16384} & -\frac{77605}{16384} &
	\frac{67353}{16384} & -\frac{43971}{16384} & 1
     \end{pmatrix}   \]

\begin{multline*}
  \tilde{D}_1^{10} = 
    \diag \Bigl( \begin{matrix}
      \frac{19 \sqrt{195734}}{139810} & \frac{57 \sqrt{10130}}{20260} & \frac{60154 \sqrt{193114380030}}{32185730005} &
      \frac{2685897 \sqrt{173245498730}}{692981994920} & \frac{32615476 \sqrt{1772129191206}}{19198066238065} \end{matrix} \\
     \begin{matrix} \frac{320784657 \sqrt{49082496344587186}}{30204613135130576} & \frac{41893024 \sqrt{8789678062}}{2113917573911} &
      \frac{8339225 \sqrt{11572967547441690}}{799643641604608} & \frac{64 \sqrt{2863718}}{214489} & \frac{115 \sqrt{356469906}}{14488576}
       \end{matrix} \Bigr) 
\end{multline*} 
\[   \begin{pmatrix} 
	0 & \frac{14}{19} & -\frac{70}{19} & \frac{176}{19} & -\frac{240}{19} & -\frac{11}{19} & \frac{559}{19} & \frac{440}{19} &
	\frac{136}{19} & 1 \\[6pt]
	-\frac{14}{57} & \frac{14}{19} & - \frac{10}{9} & \frac{154}{171} & \frac{13}{19} & -\frac{583}{171} & \frac{455}{171} & 
        \frac{385}{57} & \frac{221}{57} & 1 \\[6pt] 
	0 & -\frac{3419}{90231} & \frac{17095}{90231} & -\frac{1425347}{2887392} & \frac{774985}{962464} & -\frac{11}{19} & -\frac{23868}{30077} &
        \frac{1267475}{721848} & \frac{28101}{12664} & 1 \\[6pt]
	\frac{33352}{895299} & -\frac{33352}{298433} & \frac{1864805}{10743588} & -\frac{1847363}{10743588} & \frac{44417}{3581196} & 
        \frac{3764651}{10743588} & -\frac{1676545}{2685897} & \frac{6615}{298433} & \frac{1060307}{895299} & 1 \\[6pt] 
	0 & \frac{16997513}{1565542848} & -\frac{84987565}{1565542848} & \frac{452164559}{3131085696} & -\frac{268652845}{1043695232} &
	\frac{72951373}{260923808} & -\frac{4507737}{260923808} & -\frac{390871835}{782771424} & \frac{5774883}{13732832} & 1 \\[6pt] 
	-\frac{16858127}{641569314} & \frac{16858127}{213856438} & -\frac{480980555}{3849415884} & \frac{23247763}{167365908} &
	-\frac{387592627}{5560267388} & -\frac{6519604355}{50042406492} & \frac{383279392}{962353971} & -\frac{1713869980}{4170200541} &
	-\frac{1017012947}{4170200541} & 1 \\[6pt] 
	0 & -\frac{8340319}{670288384} & \frac{41701595}{670288384} & -\frac{903622335}{5362307072} & \frac{1727236023}{5362307072} &
	-\frac{74014413}{167572096} & \frac{28755857}{83786048} & \frac{1728420455}{10724614144} & -\frac{505997741}{564453376} & 1 \\[6pt]
	\frac{1274416}{19013433} & -\frac{1274416}{6337811} & \frac{51583246}{158445275} & -\frac{191156966}{475335825} &
	\frac{144242082}{411957715} & -\frac{14844610}{247174629} & -\frac{83114491}{158445275} & \frac{2550408241}{2059788575} &
	-\frac{9702647507}{6179365725} & 1 \\[6pt] 
	0 & \frac{663}{16384} & -\frac{3315}{16384} & \frac{9163}{16384} & -\frac{18819}{16384} & \frac{407099}{212992} & -\frac{42891}{16384} &
        \frac{1218885}{425984} & -\frac{74409}{32768} & 1 \\[6pt] 
	-\frac{4862}{6555} & \frac{4862}{2185} & -\frac{1612}{437} & \frac{33124}{6555} & -\frac{13428}{2185} & \frac{43868}{6555} &
	-\frac{13951}{2185} & \frac{2213}{437} & -\frac{19637}{6555} & 1  
     \end{pmatrix} \]}}

\bibliographystyle{siamplain}
\bibliography{multiwavelet-refs}